\documentclass[12pt,a4paper]{article}
\usepackage{amssymb, graphicx}
\usepackage{latexsym}
\usepackage{amsmath}
\usepackage{amssymb}   
\usepackage{amsthm}
\usepackage{amscd}
\usepackage{color}
\numberwithin{equation}{section}
\newtheorem{theorem}{Theorem}[section]
\newtheorem{lemma}[theorem]{Lemma}
\newtheorem{corollary}[theorem]{Corollary}
\newtheorem{proposition}[theorem]{Proposition}

\theoremstyle{definition}

\newtheorem{definition}[theorem]{Definition}

\theoremstyle{remark}

\usepackage{url,hyperref}

\begin{document}

\title{Wavelet characterization of local Muckenhoupt weighted 
Lebesgue spaces with variable exponent}

\author{Mitsuo Izuki, Toru Nogayama, Takahiro Noi and Yoshihiro Sawano}

\maketitle

\begin{abstract}
Our aim in this paper is to characterize
local Muckenhoupt weighted Lebesgue spaces with variable exponent
by compactly supported smooth wavelets.
We also investigate
necessary and sufficient conditions
for the corresponding modular inequalities to hold.
One big achievement is that the weights with exponetial growth
can be handled in the framework of variable exponents.
\end{abstract}



\section{Introduction}

Wavelets with proper decay and smoothness
give us characterizations of various function spaces.
In fact we can obtain  norms
equivalent to those spaces by using 
some square functions involving wavelet coefficients.
The first author 
\cite{IzukiGMJ}
and Kopaliani \cite{Kop} have 
initially and independently obtained the wavelet characterizations
of Lebesgue spaces with variable exponent.
Later the characterizations 
have been generalized by \cite{INS2015} 
to the Muckenhoupt weighted setting.

The theory of Lebesgue spaces with variable exponent
goes back to \cite{Orlicz31}.
After that, Nakano investigated 
Lebesgue spaces with variable exponent
in his Japanese books \cite{Nakano50,Nakano51}.
The theory of Lebesgue spaces with variable exponent
developed after
Kov$\acute{\rm{a}}\check{\rm{c}}$ik
and R$\acute{\rm{a}}$kosn\'ik
investigated Sobolev spaces with variable exponent in 90's
\cite{KR}.
Among others, Diening investigated the boundedness
of the Hardy--Littlewood maximal operator
in \cite{DieningMIA2004},
which paved the way to exhaustive investigation
of variable exponent Lebesgue spaces.
For example, 
Cruz-Uribe, SFO, Fiorenza and Neugebauer
further studied the boundedness of the Hardy--Littlewood maximal operator
in \cite{CFN2003,CFN2004}.
We refer to \cite{CF-book,INS2014}
as well as \cite[p. 447]{Sawano-text-2018}
for more details.
Moreover 
the study on generalization
of the classical Muckenhoupt weights
in terms of variable exponent
has been developed \cite{CDH-2011,CFN-2012}.
The second and fourth authors \cite{NS-local}
have defined the class of local Muckenhoupt weights
and obtained boundedness of
the local Hardy--Littlewood maximal operator
motivated by Rychkov \cite{Rychkov}.

The goal of this paper is to establish the theory
of wavelets on local weighted Lebesgue spaces
with variable exponent,
which is a follow-up of the paper \cite{NS-local}.
We seek to characterize the spaces
in terms of 
the inhomogeneous wavelet expansion.

In this paper we use the following notation of variable exponents.
Let $p(\cdot)\,:\,\mathbb{R}^n\to[1,\infty)$ be a measurable function,
and let $w$ be a weight, that is, a measurable function 
which is positive almost everywhere.
Then, we define the weighted variable Lebesgue space 
$L^{p(\cdot)}(w)$ to be the set of all measurable functions $f$ such that
for some $\lambda>0$,
\[
\int_{\mathbb{R}^n}
\left(\frac{|f(x)|}{\lambda}\right)^{p(x)}w(x)
\,{\rm d}x<\infty.
\]
When we investigate the boundedness 
of the Hardy--Littlewood maximal operator
$M$ defined by (\ref{eq:191130-1}),
the following two conditions seem standard:
\begin{itemize}
\item[$(1)$]
An exponent $r(\cdot)$ satisfies the  local log-H\"older continuity condition
if there exists $C>0$ such that
\begin{equation}\label{eq:190613-1}
{\rm L H}_0\,:\,|r(x)-r(y)|\le \frac{C}{-\log|x-y|}, 
\quad x,y \in \mathbb{R}^n, \quad |x-y|\le \frac12.
\end{equation}
The set ${\rm L H}_0$ collects all exponents $r(\cdot)$
which satisfies (\ref{eq:190613-1}). 

\item[$(2)$]
An exponent $r(\cdot)$ satisfies 
the log-H\"older continuity condition
at $\infty$ if there exist 
$C>0$ and $r_\infty \in [0,\infty)$ such that
\begin{equation}\label{eq:190613-2}
{\rm L H}_\infty\,:\,|r(x)-r_\infty|\le \frac{C}{\log(e+|x|)}, 
\quad x \in \mathbb{R}^n.
\end{equation}
The set ${\rm L H}_\infty$ collects all exponents $r(\cdot)$
which satisfies (\ref{eq:190613-2}).
\end{itemize}

We also recall the theory of wavelets.
Based on the fundamental wavelet theory
(see \cite{Dau-1988,LM,Meyer-book,Sawano-text-2018,Woj-book}),
we can construct
compactly supported $C^1$-functions
$\varphi$ and $\psi^l$ ($l=1,2, \ldots , 2^n-1$)
so that the following conditions are satisfied:
\begin{enumerate}
\item
For any $J\in \mathbb{Z}$, 
the system
\[
\left\{  \varphi_{J,k}, \, \psi^l_{j,k}  \, : \, 
k\in \mathbb{Z}^n , \, j\ge J, \, l=1,2, \ldots , 2^n-1 \right\}
\]
is an orthonormal basis of $L^2(\mathbb{R}^n)$.
Here given a 
function $F$ defined on $\mathbb{R}^n$, we write 
\[
F_{j,k}\equiv 2^{\frac{j n}{2}}F(2^j \cdot -k)
\]
for $j \in {\mathbb Z}$ and $k\in {\mathbb Z}^n$.

\item
The functions
$\varphi$ and $\psi^l$ ($l=1,2, \ldots , 2^n-1$)
belong to $C^1(\mathbb{R}^n)$.
In addition, 
they are real-valued 
and compactly supported with
$\mathrm{supp}\varphi =\mathrm{supp}\psi^l=[0,2N-1]^n$
for some $N\in \mathbb{N}$.
\end{enumerate}
We also define
$\displaystyle 
\chi_{j,k} \equiv 
2^{\frac{j n}{2}}\chi_{Q_{j,k}}
$
for $j \in {\mathbb Z}$ and $k=(k_1,k_2,\ldots,k_n) \in {\mathbb Z}^n$,
where $Q_{j,k}$ is the dyadic cube 
given by (\ref{define-Qjk}).
Then using
the $L^2$-inner product
$\langle \cdot,\cdot \rangle$,
we define the three square functions by
\begin{align*}
Vf&\equiv \left(  \sum_{k\in \mathbb{Z}^n} \left| 
\langle f,\varphi_{J,k} \rangle \varphi_{J,k} \right|^2 \right)^{\frac12} , \\
W_1 f&\equiv \left(  
\sum_{l=1}^{2^n-1}
\sum_{j=J}^\infty
\sum_{k\in \mathbb{Z}^n} \left| 
\langle f,\psi_{j,k}^l \rangle \psi_{j,k}^l \right|^2 \right)^{\frac12}, \\
W_2 f&\equiv \left(  
\sum_{l=1}^{2^n-1}
\sum_{j=J}^\infty
\sum_{k\in \mathbb{Z}^n} \left| 
\langle f,\psi_{j,k}^l \rangle \chi_{j,k} \right|^2 \right)^{\frac12} .
\end{align*}
Here $J$ is a fixed integer.

Denote by $\mathcal{Q}$ the set of all compact cubes
whose edges are parallel to coordinate axes.
We will mix the notions considered 
in \cite{CFN-2012,Rychkov}
to define the local Muckenhoupt class as follows:
\begin{definition}
Given an exponent $p(\cdot)\,:\,\mathbb{R}^n \to [1,\infty)$ and a weight $w$, 
we say that $w \in A^{\rm loc}_{p(\cdot)}$ if 
$
[w]_{A^{\rm loc}_{p(\cdot)}}
\equiv 
\sup\limits_{Q \in {\mathcal Q}, |Q|\le1} 
|Q|^{-1} \|\chi_Q\|_{L^{p(\cdot)}(w)}\|\chi_Q\|_{L^{p'(\cdot)}(\sigma)}<\infty,
$
where $\sigma \equiv w^{-\frac{1}{p(\cdot)-1}}$
and the supremum is taken over all cubes $Q \in \mathcal{Q}$
with volume less than $1$.
\end{definition}

Unlike the class of $A_{p}$ and $A_{p(\cdot)}$,
we can consider $w(x)=\exp(\alpha|x|)$ for any $\alpha \in {\mathbb R}$.
Another typical example is $w(x)=(1+|x|)^A$ for any $A \in {\mathbb R}$.

\begin{theorem}\label{thm:191125-1}
Let $p(\cdot) \in {\rm L H}_0 \cap {\rm L H}_{\infty}$
satisfy
$1<p_-
\equiv
{\rm essinf}_{x \in {\mathbb R}^n} p(x) \le p_+
\equiv
{\rm esssup}_{x \in {\mathbb R}^n} p(x) <\infty$, 
and let
$w\in A_{p(\cdot)}^{\mathrm{loc}}$.
Fix $J \in {\mathbb Z}$ arbitrarily.
Then there exists a constant $C>0$ such that
\begin{align}\label{eq:191125-4}
C^{-1} \| f \|_{L^{p(\cdot)}(w)}
\le 
\| Vf \|_{L^{p(\cdot)}(w)}  + \| W_1 f \|_{L^{p(\cdot)}(w)}
\le C\, \| f \|_{L^{p(\cdot)}(w)} ,\\
\label{eq:191125-5}
C^{-1} \| f \|_{L^{p(\cdot)}(w)}
\le 
\| Vf \|_{L^{p(\cdot)}(w)}  + \| W_2 f \|_{L^{p(\cdot)}(w)}
\le C\, \| f \|_{L^{p(\cdot)}(w)}
\end{align}
for all $f\in L^1_{\rm loc}$.
\end{theorem}
Our result extends
the one in \cite{Lemarie}
to the variable exponent setting
and the one in \cite{INS2015}
to the local weight setting.
Also, the corresponding modular inequality fails,
which can be proved in a similar way to \cite{IzukiGMJ};
see Section \ref{s5}.
Remark that $L^{p(\cdot)}(w)$ is a subset of $L^1_{\rm loc}$,
since $w \in A_{p(\cdot)}^{\rm loc}$;
see Section \ref{s2} for more.

The proof of Theorem \ref{thm:191125-1} uses
the boundedness of generalized local Calder\'{o}n--Zygmund operators.
Here and below, given a function space $X$,
$X_{\rm c}$ denotes the set of all functions $f \in X$ with compact support.
An $L^2$-bounded linear operator $T$
is a $($generalized$)$ local Calder\'{o}n--Zygmund operator
$($with the kernel $K$$)$,
if it satisfies the following conditions:
\index{generalized Calderon--Zygmund operator@(generalized) Calder\'{o}n--Zygmund operator}
\begin{enumerate}
\item[$(1)$]
There exists
$K \in L^{1}_{\rm loc}({\mathbb R}^n \times {\mathbb R}^n \setminus \{(x,x)\,:\,x \in {\mathbb R}^n\})$
such that,
for all $f \in L^2_{\rm c}({\mathbb R}^n)$, we have
\begin{equation}\label{eq:sing1}
T f(x)=\int_{{\mathbb R}^n}K(x,y)f(y){\rm d}y
\mbox{ for almost all } x \notin {\rm supp}(f).
\end{equation}
\item[$(2)$]
There exist constants $\gamma$, $D_{1}$ and $D_{2}$ such that
the two conditions below hold
for all $x,y,z \in {\mathbb R}^n$;
\begin{enumerate}
\item[$(i)$]
local size condition:
\index{size condition@size condition}
\begin{align}
\label{eq:sing3}
|K(x,y)| \le D_{1}|x-y|^{-n}\chi_{[-\gamma,\gamma]^n}(x-y)
\end{align}
if $x \ne y$,
\item[$(ii)$]
H\"{o}rmander's condition:
\index{Hormander's condition@H\"{o}rmander's condition}
\begin{align}
\label{eq:sing2}
|K(x,z)-K(y,z)|
+
|K(z,x)-K(z,y)|
\le D_{2}\frac{|x-y|}{|x-z|^{n+1}}
\end{align}
if $0<2|x-y|<|z-x|$.
\end{enumerate}
\end{enumerate}
This is an analogue of generalized singular integral operators,
which requires
\begin{align}
\label{eq:sing3a}
|K(x,y)| \le D_{1}|x-y|^{-n}
\end{align}
instead of
(\ref{eq:sing3})
if $x \ne y$.
It is known that all generalized singular integral operators,
initially defined on $L^2$, can be extended to a bounded
linear operator on $L^p$ for any $1<p<\infty$.

The structure of the remaining part of this paper is as follows:
First of all, in Section \ref{s2},
we collect some preliminary facts.
Section \ref{s3} handles 
generalized local singular integral operators
acting on Lebesgue spaces
with variable exponents.
Section \ref{s4} proves
Theorem \ref{thm:191125-1}.
In Section \ref{s5} we show that the modular inequality
fails unless $p(\cdot)$ is constant.

In this paper
we use the following notation: 
\begin{enumerate} 
\item
The set $\mathbb{N}_0\equiv\{0,1,\ldots\}$
consists of all non-negative integers.
\item
Write
\[
{\mathcal D}_j\equiv
\{Q_{j,k}\,:\,k\in \mathbb{Z}^n\}.
\]
\item
Let $Q_{0}\equiv \prod\limits_{m=1}^n
\left[ a_m,b_m\right]$ be a cube.
A dyadic cube with respect to $Q_{0}$
is the set of the form
\[
 \prod_{m=1}^n
\left[ a_m+\frac{k_m-1}{2^j}(b_m-a_m),a_m+\frac{k_m}{2^j}(b_m-a_m)\right]
\]
for some
$j \in {\mathbb N}_0$ and
$k=(k_1,k_2,\ldots,k_n) \in \{1,2,\ldots,2^j\}^n$.
The set
${\mathcal D}(Q_{0})$
collects all dyadic cubes
with respect to a cube $Q_{0}$.
\index{${\mathcal D}(Q_{0})$}
\item
The letter $C$ denotes positive constants
that may change from one occurrence to another.
Let $A,B \ge 0$.
Then $A \lesssim B$ means
that there exists a constant $C>0$
such that $A \le C B$,
where $C$ depends only on the parameters
of importance.
The symbol $A \sim B$ means
that $A \lesssim B$ and $B \lesssim A$
happen simultaneously.
\item
Let $\mathcal{M}$ be the set of all complex-valued measurable functions 
defined on $\mathbb{R}^n$.
Likewise, for a measurable set $E$,
let $\mathcal{M}(E)$ be the set of all complex-valued measurable functions 
defined on $E$.
\item
Let $E$ be a set.
Then we denote its indicator function
by $\chi_E$. 
\item
The symbol $\langle f,g \rangle$
stands for the $L^2$-inner product. That is, we write
\[
\langle f,g \rangle \equiv \int_{\mathbb{R}^n}
f(x) \overline{g(x)} \, {\rm d}x
\]
for all complex-valued measurable
$L^2$-functions $f, \,g $ defined on $\mathbb{R}^n$.
\item
A set $S$ is said to be a dyadic cube if 
\begin{equation}
S=Q_{j,k}\equiv  \prod_{m=1}^n
\left[  2^{-j}k_m , 2^{-j}(k_m+1) \right]
\label{define-Qjk}
\end{equation}
for some $j\in \mathbb{Z}$ and $k=(k_1,k_2, \ldots , k_n )\in \mathbb{Z}^n$.
\item
The set $\mathcal{P}$ consists of all
$p(\cdot):\mathbb{R}^n \to [1,\infty)$
such that $1<p_-\le p_+<\infty$.
\item
Given a cube $Q$,
we denote by $c(Q)$ the center of $Q$
and by $\ell(Q)$ the sidelength of $Q$:
$\ell(Q)=|Q|^{\frac1n}$,
where $|Q|$ denotes the volume of the cube $Q$.
In addition, $|E|$ is the Lebesgue measure
for general measurable set $E\subset \mathbb{R}^n$.

\item
Let $f$ be a measurable function.
We consider 
the local maximal operator given by
\[
M^{\rm loc}f(x)
\equiv
\sup_{Q \in \mathcal{Q}, |Q|\le1} \frac{\chi_{Q}(x)}{|Q|}\int_{Q}|f(y)| {\rm d}y
\quad (x \in {\mathbb R}^n).
\]
Needless to say, this is an analogue of the 
Hardy--Littlewood maximal operator given by
\begin{equation}\label{eq:191130-1}
Mf(x)
\equiv
\sup_{Q \in \mathcal{Q}} \frac{\chi_{Q}(x)}{|Q|}\int_{Q}|f(y)| {\rm d}y
\quad (x \in {\mathbb R}^n).
\end{equation}
\item
Let $K \in {\mathbb N}$.
The operator
$(M^{\rm loc})^{K}$ is the $K$-fold composition
of $M^{\rm loc}$.
\item
Let $E$ be a measurable set in ${\mathbb R}^n$.
For a function
$f:E \to {\mathbb C}$,
$f^*$ denotes its decreasing rearrangement.
\end{enumerate}

\section{Preliminaries}
\label{s2}

Here we collect some preliminary facts
used in this paper.
First, we recall the structure of weighed
Banach function spaces
and then we concentrate on
weighted Lebesgue spaces with variable exponent.

\subsection{Definition of function spaces}

Let $p(\cdot)\, : \, \mathbb{R}^n \to [1,\infty)$ be a measurable function, 
and let $w$ be a weight.
We have defined the weighted Lebesgue space 
$L^{p(\cdot)}(w)$ with variable exponent $p(\cdot)$
in Introduction. In addition, 
the space $L^{p(\cdot)}(w)$ is a Banach space equipped with the norm
given by
\[
\| f \|_{L^{p(\cdot)}(w)}\equiv
\inf \left\{  \lambda >0 \, : \, 
\int_{\mathbb{R}^n}  \left( \frac{|f(x)|}{\lambda} \right)^{p(x)}w(x)
{\rm d}x \le 1
\right\} .
\]
If $w=1$ almost everywhere,
then $L^{p(\cdot)}(w)$ is a non-weighted variable Lebesgue
and we write $L^{p(\cdot)}\equiv   L^{p(\cdot)}(w)$.
Moreover if $p(\cdot)$ equals to a constant $p$, then
$L^{p(\cdot)}=L^p$, that is the usual $L^p$ space.
When we consider non-weighted function spaces defined on $\mathbb{R}^n$,
we may simply write
$L^2\equiv L^2(1)$,
$L^1_{{\rm loc}} \equiv L^1_{{\rm loc}}(1)$ and so on.

\subsection{Weighted Banach function spaces}

We define Banach function spaces 
and then state their fundamental properties.
For further information on the theory of Banach function space 
including the proof of Lemma \ref{Lemma-Banach-associate} 
below we refer to \cite{BS-book,ORS}.
We additionally show some properties 
of Banach function spaces in terms of boundedness of
the Hardy--Littlewood maximal operator. 
We will also consider the weighted case 
based on Karlovich and Spitkovsky \cite{KarlovichSpitkovsky}.

\begin{definition}
Let
$X$ 
be
a linear subspace of $\mathcal{M}$. 
\begin{enumerate}
\item
The space $X$ is said to be a Banach function space if there exists a functional 
$\| \cdot \|_X \, : \, \mathcal{M} \to [0,\infty]$ 
satisfying the following properties: 
Let $f, \, g, \, h, \, f_j \in \mathcal{M}$ $(j=1, \, 2, \, \ldots)$
and 
$\lambda \in {\mathbb C}$ be arbitrary.  
\begin{enumerate} 
\item
$f\in X$ holds if and only if $\| f \|_X<\infty$.
\item Norm property:
  \begin{enumerate}
  \item Positivity: $\| f \|_X \ge 0$.
  \item Strict positivity: $\| f \|_X=0$ holds if and only if $f(x)=0$ for almost every 
                                  $x\in \mathbb{R}^n$.
  \item Homogeneity: $\| \lambda f \|_X=|\lambda| \cdot \| f \|_X$ holds.
  \item Triangle inequality: $\| f+g \|_X \le \| f \|_X +\| g \|_X$.
  \end{enumerate}

\item Symmetry: $\| f \|_X= \| |f| \|_X$.
\item Lattice property: If $0\le g(x)\le f(x)$ for almost every $x\in \mathbb{R}^n$, 
         then $\| g \|_X \le \| f \|_X$.
\item Fatou property: 
If $0\le f_j(x)\le f_{j+1}(x)$ for all $j$ 
and $f_j(x)\to f(x)$ as $j\to \infty$
for 
almost every $x\in \mathbb{R}^n$, then 
$\displaystyle{\lim_{j\to \infty} \| f_j \|_X  =\| f \|_X}$. 

\item
For every measurable set $E\subset \mathbb{R}^n$ with $|E|<\infty$, 
$\| \chi_E \|_X$ is finite. Additionally there exists a constant $C_E>0$ 
depending only on $E$ such that 
\[
\int_E |h(x)|\, \mathrm{d}x \le C_E \| h \|_X.
\]
\end{enumerate}

\item
Suppose that $X$ is a Banach function space 
equipped with a norm $\| \cdot \|_X$.
The associate space $X'$ is defined by
$
X'\equiv \{  f\in \mathcal{M} \, : \, \| f \|_{X'}<\infty  \} $,
where
$\displaystyle 
\| f \|_{X'}\equiv \sup \left\{
\left| \int_{\mathbb{R}^n} f(x)g(x)\, \mathrm{d}x \right|
\, : \, \| g \|_X\le 1
\right\} $.
\end{enumerate}
\end{definition}

\begin{lemma}
\label{Lemma-Banach-associate}
Let $X$ be a Banach function space.
\begin{enumerate}
\item $($The Lorentz--Luxemberg theorem.$)$ 
The associate space $X'$ is again a Banach function space,
whose associate space $X''$ coincides with $X$,
in particular, 
the norms $\| \cdot \|_{X''}$ and $\| \cdot \|_X$ are 
equivalent. 
\item $($The generalized H\"older inequality.$)$
\[
\|f \cdot g\|_{1}
\le \| f \|_X \| g \|_{X'},
\] 
whenever $f\in X$ and $g\in X'$.
\end{enumerate}
\end{lemma}

Kov$\acute{\rm{a}}\check{\rm{c}}$ik and R$\acute{\rm{a}}$kosn\'ik 
\cite{KR} have proved that 
the  Lebesgue space
$L^{p(\cdot)}$ with variable exponent 
$p(\cdot)$
is a Banach function space and the associate space equals to 
$L^{p'(\cdot)}$
with norm equivalence.

Below we define weighted Banach function spaces
and give some of their properties. 
Let $X$ be a Banach function space. 
The set $X_{\mathrm{loc}}(\mathbb{R}^n)$ consists of all $f \in {\mathcal M}$ 
such that $f\chi_{[-j,j]^n} \in X$ for $j \in {\mathbb N}$.  
Given a function $W$ satisfying
$0<W(x)<\infty$ for almost every $x\in \mathbb{R}^n$, 
$W\in X_{\mathrm{loc}}(\mathbb{R}^n)$ 
and $W^{-1} \in (X')_{\mathrm{loc}}(\mathbb{R}^n)$, 
we define the weighted Banach function space
$X(\mathbb{R}^n,W)$ by
\[
X(\mathbb{R}^n,W)\equiv 
\left\{ f\in \mathcal{M}
\, : \, fW \in X
\right\} .
\]
We summarize the properties of
the weighted Banach function space $X(\mathbb{R}^n,W)$.
\begin{lemma} 
\label{Lemma-weighted-Banach}
Let $W$ be a weight as above, and let $X$ be a
Banach function space.
\begin{enumerate}
\item
The weighted Banach function space $X(\mathbb{R}^n,W)$,
which is equipped with the norm
$\| f \|_{X(\mathbb{R}^n,W)}\equiv \| f W \|_{X}$
for $f \in {\mathcal M}$, 
is a Banach function space.
\item
The associate space of $X(\mathbb{R}^n,W)$ 
coincides with $X'(\mathbb{R}^n,W^{-1})$. 
\end{enumerate}
\end{lemma}

The properties above naturally arise 
from those of usual Banach function spaces 
and their proofs are found in \cite{KarlovichSpitkovsky}.

\subsection{Some observations on weighted Lebesgue spaces with variable exponents}

First,
we apply Lemma \ref{Lemma-weighted-Banach}
to local weighted Lebesgue spaces with variable exponent.
We use the following duality principle.

\begin{proposition}
Let $p(\cdot)\in \mathcal{P}$
and $w$ be a weight
whose dual weight $\sigma= w^{-\frac{1}{p(\cdot)-1}}$ is locally integrable.
Then $L^{p(\cdot)}(w)$
is a Banach function space
and the associate space 
$\left(L^{p(\cdot)}(w)\right)'$
is $L^{p'(\cdot)}(\sigma)$
with norm equivalence.
\end{proposition}

\begin{proof}
As we mentioned,
the associate space of $L^{p(\cdot)}$
is $L^{p'(\cdot)}$
with norm equivalence
\cite{KR}.
Since
$f \in L^{p(\cdot)}(w) \mapsto f \cdot w^{\frac{1}{p(\cdot)}} \in L^{p(\cdot)}$
and
$f \in L^{p'(\cdot)}(\sigma) \mapsto f \cdot \sigma^{\frac{1}{p'(\cdot)}} \in L^{p'(\cdot)}$
are isomorphisms,
the conclusions are clear.
\end{proof}




In \cite{NS-local},
the maximal inequality
(Proposition \ref{thm:main})
and
its vector-valued extension
(Proposition \ref{thm 191031-3})
were obtained.
\begin{proposition}\label{thm:main}{\rm \cite{NS-local}}
Suppose that $p(\cdot) \in {\mathcal P}
\cap {\rm L H}_0 \cap {\rm L H}_{\infty}$.
Then given any $w \in A_{p(\cdot)}^{\rm loc}$,
there exists a constant $C>0$ such that
\[
\|M^{\rm loc}_{}f\|_{L^{p(\cdot)}(w)} \le C \|f\|_{L^{p(\cdot)}(w)}
\]
for all $f \in L^{p(\cdot)}(w)$.
\end{proposition}
This result corresponds to the ones 
obtained in \cite{CDH-2011,CFMP}.

By adapting an extrapolation result in \cite{CMP-book}
to our local weight setting,
the second and fourth authors obtained the following vector-valued inequality.
\begin{proposition} \label{thm 191031-3}{\rm \cite{NS-local}}
Suppose that $p(\cdot) \in {\mathcal P}
\cap {\rm L H}_0 \cap {\rm L H}_{\infty}$.
Let also $w \in A_{p(\cdot)}^{\rm loc}$ and $1<q \le \infty$.
Then
 for any sequence $\{f_j\}_{j=1}^\infty \subset {\mathcal M}$,
we have 
\begin{align} \label{eq 191031-4a}
\left\|
\left(
\sum_{j=1}^\infty
\left[M^{\rm loc}f_j{}\right]^q
\right)^{\frac{1}{q}}
\right\|_{L^{p(\cdot)}(w)}
\le C
\left\|
\left(
\sum_{j=1}^\infty
|f_j|^q
\right)^{\frac{1}{q}}
\right\|_{L^{p(\cdot)}(w)}.
\end{align}
A natural modification is made when $q=\infty$.
\end{proposition}

Let $E\subset \mathbb{R}^n$ be a measurable set.
The set  
$C^\infty_{\rm c}(E)$ consists of all infinitely differentiable functions defined on 
$\mathbb{R}^n$
whose support is compact and contained in $E$.
Once we obtain the boundedness
of $M^{\rm loc}$,
with ease
we can obtain the density of
$C^\infty_{\rm c}(\mathbb{R}^n)$.

\begin{corollary}
Suppose that $p(\cdot) \in {\mathcal P}
\cap {\rm L H}_0 \cap {\rm L H}_{\infty}
$,
and let $w \in A_{p(\cdot)}^{\rm loc}$.
Then
$C^\infty_{\rm c}(\mathbb{R}^n)$
is dense in $L^{p(\cdot)}(w)$.
\end{corollary}

Remark that this is an analogue
of \cite[Theorem 2.8]{INS2015}.

\begin{proof}
By the Lebesgue convergence theorem
obtained in \cite{INS2014},
we have only to approximate functions
in $L^\infty_{\rm c}$,
the set of all essentially bounded functions with compact support.
Let $f \in L^\infty_{\rm c}$.
Choose a non-negative function
$\tau \in C^\infty_{\rm c}([-1,1]^n)$
with $L^1$-norm $1$ and consider
$j^n\tau(j\cdot)*f$
for each $j \in {\mathbb N}$.
Then we know that
$|j^n\tau(j\cdot)*f| \lesssim M^{\rm loc} f$
for all $j \in \mathbb{N}$.
We know that $j^n\tau(j\cdot)*f(x) \to f(x)$
for almost all $x \in {\mathbb R}^n$
as $j \to \infty$
by the Lebesgue differentiation theorem.
Thus, once again we can use the Lebesgue convergence theorem
mentioned above to have
$j^n\tau(j\cdot)*f \to f$
as $j \to \infty$.
\end{proof}
\section{Generalized local singular integral operators}
\label{s3}

We prove that the generalized local singular integral operators
are bounded on $L^{p(\cdot)}(w)$ if the postulates
in
Proposition \ref{thm:main}
are satisfied.
There are several ways to prove
the boundedness of  the generalized (local) 
singular integral operators.
One technique is to transform the sharp maximal inequality
obtained by Fefferman and Stein \cite{FeSt72}
to the form adapted to our function spaces.
Here, we use the modified version considered
by Hyt\"{o}nen \cite{Hytonen14} and Lerner \cite{Lerner13-2}.
For $Q \in {\mathcal Q}$,
the median of $f \in {\mathcal M}(Q)$, which is close to 
$(\chi_{Q} f)^*\left(2^{-1}|Q|\right)$,
will be an important role in the sequel.
We use the following notation:
The mean oscillation
of $f$ over a cube $Q$
of level $\lambda \in(0,1)$ is given by
$
\omega_\lambda(f;Q)
\equiv
\inf\limits_{c \in {\mathbb C}}
((f-c)\chi_{Q})^*(\lambda|Q|).
$

\begin{definition}[Median]\label{defi:150824-41}
\index{median@median}
\index{${\rm Med}(f;Q)$}
Let $Q \in {\mathcal Q}$
and $f \in {\mathcal M}(Q)$ be a real-valued function.
Define
$
{\rm MED}(f;Q)\equiv\{a \in {\mathbb R}\,:\,
a\mbox{ satisfies (\ref{eq:131112-23})}\},
$
where condition $(\ref{eq:131112-23})$ is given by
\begin{equation}\label{eq:131112-23}
|\{x \in Q\,:\,f(x)>a\}|, \quad
|\{x \in Q\,:\,f(x)<a\}|
\le \frac{1}2|Q|.
\end{equation}
Denote by ${\rm Med}(f;Q)$
any element in ${\rm MED}(f;Q)$.
The quantity ${\rm Med}(f;Q)$ is called the median of $f$
over $Q$.
One can regard ${\rm Med}(f;Q)$
as if it were a mapping
$Q \mapsto {\rm Med}(f;Q)$.
\end{definition}

We now recall the notion of sparseness.
A set of cubes ${\mathfrak A}$ is 
sparse,
if
there exists a disjoint collection
$\{K(Q)\}_{Q\in{\mathfrak A}}$ 
such that
$K(Q)$ contained in $Q$ and 
$2|K(Q)| \ge |Q|$
for each $Q\in{\mathfrak A}$.
Each $K(Q)$ is called
the nutshell of $Q$.
We invoke the following result obtained by Hyt\"{o}nen \cite{Hytonen14} and Lerner \cite{Lerner13-2}.

\begin{lemma}\label{lem:LH}
Let
$Q \in {\mathcal D}$,
and let
$g \in {\mathcal M}(Q)$ be a real-valued function.
Then
there exists a sparse family
${\mathcal S}(Q)
\subset {\mathcal D}(Q)$,
which depends on $g$,
such that
$Q \in {\mathcal S}(Q)$ and
\[
\chi_{Q}|g-{\rm Med}(g;Q)|
\le 
\sum_{S \in {\mathcal S}(Q)}
\omega_{2^{-n-2}}(g;S)\chi_{S}
\]
almost everywhere.
\end{lemma}

We estimate
the mean oscillation of $T f$ 
in the next lemma.
\begin{lemma}\label{lem:191130-2}
Let $T$ be a generalized local singular integral operator,
whose kernel $K$ is supported on 
$\{(x,y)\,:\,x-y \in[-\gamma,\gamma]^n\}$
with some $\gamma \in {\mathbb N}$.
Then for any cube $Q$ with $|Q| \le 1$ and
$f \in L^2$,
we have
\[
\omega_{2^{-n-2}}(T f;Q) 
\lesssim \inf\limits_{y \in Q}(M^{\rm loc})^{2\gamma+3} f(y).
\]
\end{lemma}

\begin{proof}
Let $Q(x,R)$ denote the cube of length $2R$ centered at $x$ in ${\mathbb R}^n$
in the proof.
First of all, we note that
\[
T f(x)=T[\chi_{Q(c(Q),\gamma+1)}f](x)
\]
for all $x \in Q$.
With this in mind, we set
\[
\alpha\equiv\int_{Q(c(Q),\gamma+1) \setminus 3Q}K(c(Q),y)f(y){\rm d}y.
\]
Then 
\begin{align*}
|T f(x)-\alpha|
&\lesssim
\int_{Q(c(Q),\gamma+1) \setminus 3Q}|K(x,y)-K(c(Q),y))||f(y)|dy
+|T[\chi_{3Q}f](x)|\\ 
&\lesssim
\int_{Q(c(Q),\gamma+1) \setminus 3Q}
\frac{\ell(Q)}{|z-c(Q)|^{n+1}}|f(z)|{\rm d}z
+|T[\chi_{3Q}f](x)|.
\end{align*}
Let $J_0 \in {\mathbb N}$ be the smallest integer
such that $2^{J_0}Q \supset Q(c(Q),\gamma+1)$.
We estimate
\begin{align*}
\int_{Q(c(Q),\gamma+1) \setminus 3Q}
\frac{\ell(Q)}{|z-c(Q)|^{n+1}}|f(z)|{\rm d}z
&\le
\sum_{j=1}^{J_0}
\int_{2^j Q \setminus 2^{j-1}Q}
\frac{\ell(Q)}{|z-c(Q)|^{n+1}}|f(z)|{\rm d}z\\
&\lesssim
\sum_{j=1}^{J_0}
\frac{1}{2^{j(n+1)}|Q|}
\int_{2^j Q}|f(z)|{\rm d}z.
\end{align*}
Note that
\[
\frac{1}{2^{j n}|Q|}
\int_{2^j Q}|f(z)|{\rm d}z
\lesssim
\inf_{y \in Q}
(M^{\rm loc})^{2\gamma+3} f(y)
\]
for all $j=1,2,\ldots,J_0$.
Meanwhile,
\[
\omega(T[\chi_{3Q}f];Q) \lesssim \dfrac{1}{|Q|}\int_Q|f(x)|{\rm d}x\lesssim
\inf_{y \in Q}
(M^{\rm loc})^{2\gamma+3} f(y)
\]
by virtue of the weak-$L^1$ boundedness of $T$.
Inserting these estimates into the above chain of inequalities,
we obtain
the desired result.
\end{proof}

We will prove the boundedness
of
generalized local Calder\'{o}n--Zygmund operators,
based on Lemma \ref{lem:LH}.

\begin{theorem}\label{thm:191125-2}
Suppose that $p(\cdot) \in {\mathcal P}
\cap {\rm L H}_0 \cap {\rm L H}_{\infty}$.
Let $T$ be a
generalized local Calder\'{o}n--Zygmund operator.
Let also $w \in A_{p(\cdot)}^{\rm loc}$.
Then
$T$, initially defined for $L^\infty_{\rm c}$,
can be extended to a bounded linear operator
on $L^{p(\cdot)}(w)$.
\end{theorem}
\begin{proof}
Let $f \in L^\infty_{\rm c}$.
By decomposing the integral kernel and the function $f$
we may assume that these are real-valued.
Let 
${\mathcal D}_0$
be the set of all dyadic cubes of volume $1$.
We decompose
\begin{align*}
|T f(x)|
&=
\sum_{Q \in {\mathcal D}_0}
|T f(x)|\chi_Q(x)\\
&\le
\sum_{Q \in {\mathcal D}_0}
(|T f(x)-{\rm Med}(T f;Q)|+|{\rm Med}(T f;Q)|)\chi_Q(x).
\end{align*}
Fix $Q \in {\mathcal D}_0$ for the time being.
For the first term, we use the Lerner--Hyt\"{o}nen decomposition
to have a sparse family ${\mathcal S}(Q)$ such that
\[
|T f(x)-{\rm Med}(T f;Q)|
\le 
\sum_{S \in {\mathcal S}(Q)}
\omega_{2^{-n-2}}(T f;S)\chi_{S}(x)
\]
for almost all $x \in {\mathbb R}^n$.
Since $T$ is a generalized local singular integral operator,
there exists $\gamma \in {\mathbb N}$ such that
the kernel is supported on
$\{(x,y)\,:\,x-y \in[-\gamma,\gamma]^n\}$.
For this $\gamma \in {\mathbb N}$,
\[
\omega_{2^{-n-2}}(T f;S) 
\lesssim \inf\limits_{y \in S}(M^{\rm loc})^{2\gamma+3} f(y).
\]
thanks to Lemma \ref{lem:191130-2}.
Meanwhile, by the definition of median
and the weak $L^1$-boundedness of $T$,
we have
\begin{align*}
|{\rm Med}(T f;Q)|
&=
|{\rm Med}(T[\chi_{(\gamma+1)Q} f];Q)|\\
&\lesssim
\frac{1}{|(\gamma+1)Q|}\int_{(\gamma+1)Q}|f(x)|{\rm d}x\\
&\lesssim
\inf\limits_{y \in Q}(M^{\rm loc})^{2\gamma+3} f(y).
\end{align*}
We define
\[
K(S)\equiv S \setminus \bigcup_{R \in {\mathcal S}(Q)}R.
\]
Then we have $2|K(S)|\ge|S|$.
Putting these observations all together,
we obtain
\begin{align*}
|T f(x)|
&\lesssim
\sum_{Q \in {\mathcal D}_0}
\sum_{S \in {\mathcal S}(Q)}
\inf\limits_{y \in S}(M^{\rm loc})^{2\gamma+3}f(y)
\chi_{S}(x)+
\sum_{Q \in {\mathcal D}_0}
\inf\limits_{y \in Q}(M^{\rm loc})^{2\gamma+3} f(y)
\chi_{Q}(x)\\
&\lesssim
\sum_{Q \in {\mathcal D}_0}
\sum_{S \in {\mathcal S}(Q)}
\inf\limits_{y \in S}(M^{\rm loc})^{2\gamma+3}f(y)
M^{\rm loc}\chi_{K(S)}(x)^\theta
+
(M^{\rm loc})^{2\gamma+3} f(x)
\end{align*}
for any $\theta>0$.
Thus, choosing $\theta>0$ suitably,
by virtue of Proposition \ref{thm 191031-3}
we obtain
\begin{align*}
\lefteqn{
\|T f\|_{L^{p(\cdot)}(w)}
}\\
&\lesssim
\left\|
\sum_{Q \in {\mathcal D}_0}
\sum_{S \in {\mathcal S}(Q)}
\inf\limits_{y \in S}(M^{\rm loc})^{2\gamma+3}f(y)
(M^{\rm loc}\chi_{K(S)})^\theta
\right\|_{L^{p(\cdot)}(w)}
+
\|(M^{\rm loc})^{2\gamma+3}f\|_{L^{p(\cdot)}(w)}\\
&\lesssim
\left\|
\sum_{Q \in {\mathcal D}_0}
\sum_{S \in {\mathcal S}(Q)}
\inf\limits_{y \in S}(M^{\rm loc})^{2\gamma+3}f(y)
\chi_{K(S)}
\right\|_{L^{p(\cdot)}(w)}
+
\|(M^{\rm loc})^{2\gamma+3}f\|_{L^{p(\cdot)}(w)}\\
&\lesssim
\|(M^{\rm loc})^{2\gamma+3}f\|_{L^{p(\cdot)}(w)}\\
&\lesssim
\|f\|_{L^{p(\cdot)}(w)},
\end{align*}
as required.
\end{proof}

\section{Proof of Theorem \ref{thm:191125-1}}
\label{s4}

We will suppose that $f$ is a real-valued function
keeping in mind that $\varphi$ and $\psi$
are real-valued.
Once we prove the right-hand estimate,
then the left-hand estimate is automatically 
obtained by a duality argument.
In fact,
we have
\begin{align*}
\int_{{\mathbb R}^n}f(x)g(x){\rm d}x
&=
\sum_{k \in {\mathbb Z}^n}
\langle f,\varphi_{J,k} \rangle
{\langle g,\varphi_{J,k} \rangle}
+
\sum_{l=1}^{2^n-1}
\sum_{j=J}^\infty
\sum_{k\in \mathbb{Z}^n}
\langle f,\psi_{j,k}^l \rangle 
{\langle g,\psi_{j,k}^l \rangle}
\end{align*}
for all real-valued functions
$f,g \in L^\infty_{\rm c}$
and hence by using H\"{o}lder's inequality twice, we have
\begin{align*} 
\left|
\int_{{\mathbb R}^n}f(x){g(x)}{\rm d}x
\right|
\lesssim
\left\|V f
\right\|_{L^{p(\cdot)}(w)}
\left\|V g
\right\|_{L^{p'(\cdot)}(\sigma)}
+
\left\|W_1f
\right\|_{L^{p(\cdot)}(w)}
\left\|W_1g
\right\|_{L^{p'(\cdot)}(\sigma)}.
\end{align*}
Since we know that
\begin{eqnarray*}
\left\|V g
\right\|_{L^{p'(\cdot)}(\sigma)}
+
\left\|W_1g
\right\|_{L^{p'(\cdot)}(\sigma)}
\lesssim
\|g\|_{L^{p'(\cdot)}(\sigma)},
\end{eqnarray*}
we obtain 
\[
\|f\|_{L^{p(\cdot)}(w)}
\lesssim
\|V f\|_{L^{p(\cdot)}(w)}
+\|W_1f\|_{L^{p(\cdot)}(w)}.
\]
A similar argument also yields
\[
\|f\|_{L^{p(\cdot)}(w)}
\lesssim
\|V f\|_{L^{p(\cdot)}(w)}
+\|W_2f\|_{L^{p(\cdot)}(w)}.
\]
By a simple limiting argument,
we may assume $f \in L^\infty_{\rm c}$
to prove
\[
\|V f\|_{L^{p(\cdot)}(w)}
+\|W_1f\|_{L^{p(\cdot)}(w)}
\lesssim
\|f\|_{L^{p(\cdot)}(w)}.
\]

Note that 
$$\left|  \mathrm{supp}\varphi_{J,k} \right|
= \left|  \prod_{m=1}^n \left[  2^{-J}k_m,2^{-J}(2N-1+k_m)  \right] \right|
= \left(  2^{-J}(2N-1) \right)^n.
$$
Then we see that
\[
Vf(x)\lesssim (M^{\mathrm{loc}})^{4N+10}f(x).
\]
Thus,
it follows from Proposition \ref{thm:main}
that
\[
\|V f\|_{L^{p(\cdot)}(w)}
\lesssim
\|f\|_{L^{p(\cdot)}(w)}.
\]

We will show 
\begin{equation}\label{eq:191125-3}
\|W_1 f\|_{L^{p(\cdot)}(w)}
\le C
\|f\|_{L^{p(\cdot)}(w)}.
\end{equation}
Let
$j \in {\mathbb Z} \cap [J,\infty)$
and
$k \in {\mathbb Z}^n$.
Since
\[
2^{j n}\int_{Q_{j,k'}}|\varphi_{j,k}(x)|\,{\rm d}x \ge C2^{\frac{j n}{2}}
\]
as long as 
$k' \in {\mathbb Z}^n$
satisfies
$Q_{j,k'} \cap Q_{j,k} \ne \emptyset$,
we have $\chi_{j,k} \le C(M^{\mathrm{loc}})^{4N+10}\varphi_{j,k}$.
Thus, the proof of 
\[
\|W_2 f\|_{L^{p(\cdot)}(w)}
\le C
\|f\|_{L^{p(\cdot)}(w)}
\]
follows immediately
from
\[
\|W_1 f\|_{L^{p(\cdot)}(w)}
\le C
\|f\|_{L^{p(\cdot)}(w)}.
\]
We remark that
\[
\left\|\left(
\sum_{l=1}^{2^n-1}
\sum_{j=J}^\infty
\sum_{k\in \mathbb{Z}^n} \left| 
\langle f,\psi_{j,k}^l \rangle \psi_{j,k}^l \right|^2 \right)^{\frac12}
\right\|_{L^{p(\cdot)}(w)}
\le C
\|f\|_{L^{p(\cdot)}(w)}.
\]
Let
$\Theta=(\Theta_1,\Theta_2,\Theta_3):{\mathbb N} \to
({\mathbb N} \cap [1,2^n-1])
\times
({\mathbb N} \cap [J,\infty))
\times
{\mathbb Z}^n$
be a bijection.
Then we have
\begin{align*}
\lefteqn{
\left(
\sum_{l=1}^{2^n-1}
\sum_{j=J}^\infty
\sum_{k\in \mathbb{Z}^n} \left| 
\langle f,\psi_{j,k}^l \rangle \psi_{j,k}^l \right|^2 \right)^{\frac12}
}\\
&=
\left(
\sum_{\mu=1}^\infty
\left| 
\langle f,\psi_{\Theta_1(\mu),\Theta_2(\mu)}^{\Theta_3(\mu)} \rangle 
\psi_{\Theta_1(\mu),\Theta_2(\mu)}^{\Theta_3(\mu)} \right|^2 \right)^{\frac12}\\
&\lesssim
\left(
\frac{1}{\pi}
\int_0^{2\pi}
\left|
\sum_{\mu=1}^\infty
{\rm sgn}(\sin(2^\mu t))
\langle f,\psi_{\Theta_1(\mu),\Theta_2(\mu)}^{\Theta_3(\mu)} \rangle 
\psi_{\Theta_1(\mu),\Theta_2(\mu)}^{\Theta_3(\mu)}\right|^{p_-}
{\rm d}t \right)^{\frac{1}{p_-}}
\end{align*}
thanks to the property of the Rademacher sequence
$\{{\rm sgn}(\sin(2^\mu \cdot))\}_{\mu=1}^\infty$
\cite{Sawano-text-2018}.
Thus, by Minkowski's inequality, we have only to show that
\[
f \mapsto 
\sum_{\mu=1}^\infty
{\rm sgn}(\sin(2^\mu t))
\langle f,\psi_{\Theta_1(\mu),\Theta_2(\mu)}^{\Theta_3(\mu)} \rangle 
\psi_{\Theta_1(\mu),\Theta_2(\mu)}^{\Theta_3(\mu)}
\]
is a generalized local singular integral operator.
However, it is a generalized singular integral operator
with the constants independent of $t$ and $l$
from the result \cite{Woj-book}.
Since the functions
$\varphi$ and $\psi^l$ 
 ($l=1,2, \ldots , 2^n-1$)
are compactly supported,
it follows that the operator in question is a local  generalized singular integral operator
with the constants independent of $t$ and $l$.
Consequently, we obtain
(\ref{eq:191125-3}) by virtue of Theorem \ref{thm:191125-2}.

\section{Modular inequalities for wavelet characterizations}\label{s5}

Similar to \cite[Theorem 4.3]{INS2015},
we can prove that
some modular inequalities for wavelet characterizations fail  
unless the exponent is constant.

\begin{theorem}
Suppose that $p(\cdot) \in {\mathcal P}
\cap {\rm L H}_0 \cap {\rm L H}_{\infty}$.
Let $w \in A_{p(\cdot)}^{\rm loc}$ and $j^* \in {\mathbb Z}$.
Then the following four conditions are equivalent:
\begin{enumerate}
\item[$\mathrm{(X1)}$]
$p(\cdot)$ is a constant.
\item[$\mathrm{(X2)}$]
For 
all $f \in L^{p(\cdot)}(w)$,
\[
\int_{{\mathbb R}^n}
\left| 
\sum_{l=1}^{2^n-1}
\sum_{k\in \mathbb{Z}^n} 
\langle f,\psi_{j^*,k}^l \rangle \psi_{j^*,k}^l (x)
\right|^{p(x)}w(x){\rm d}x
\lesssim
\int_{{\mathbb R}^n}|f(x)|^{p(x)}w(x){\rm d}x.
\]

\item[$\mathrm{(X3)}$]
For all $f \in L^{p(\cdot)}(w)$,
\[
\int_{{\mathbb R}^n}V f(x)^{p(x)}w(x){\rm d}x
+
\int_{{\mathbb R}^n}W_1 f(x)^{p(x)}w(x){\rm d}x
\sim
\int_{{\mathbb R}^n}|f(x)|^{p(x)}w(x){\rm d}x.
\]
\item[$\mathrm{(X4)}$]
For all $f \in L^{p(\cdot)}(w)$,
\[
\int_{{\mathbb R}^n}V f(x)^{p(x)}w(x){\rm d}x
+
\int_{{\mathbb R}^n}W_2 f(x)^{p(x)}w(x){\rm d}x
\sim
\int_{{\mathbb R}^n}|f(x)|^{p(x)}w(x){\rm d}x.
\]
\end{enumerate}
\end{theorem}

Fix $j^* \in {\mathbb Z}$.
Remark that
\[
f \mapsto
\sum_{l=1}^{2^n-1}
\sum_{k\in \mathbb{Z}^n} 
\langle f,\psi_{j^*,k}^l \rangle \psi_{j^*,k}^l (x)
\]
is the orthogonal projection from $L^2$ to a linear space spanned by
\[
\{\psi_{j^*,k}^l\,:\,l=1,2,\ldots,2^n-1, k \in {\mathbb Z}^n\}.
\]

\begin{proof}
From Theorem \ref{thm:191125-1} or the result
in \cite{Lemarie},
(X2)--(X4) hold once we assume (X1) holds.
Conditions (X3) and (X4) are clearly stronger than (X2).
We can prove that (X2) implies (X1)
as we did in \cite[Theorem 4.3]{INS2015}.
\end{proof}

\section*{Acknowledgements} 
\noindent
Mitsuo Izuki was partially supported by Grand-in-Aid for Scientific Research (C), No.\,15K04928, for Japan Society for the Promotion of Science. 
Takahiro Noi was partially supported by Grand-in-Aid for Young Scientists (B), No.\,17K14207, for Japan Society for the Promotion of Science. 
Yoshihiro Sawano was partially supported by Grand-in-Aid for Scientific Research (C), No.\,19K03546, for Japan Society for the Promotion of Science. 
This work was partly supported by Osaka City University Advanced Mathematical Institute (MEXT Joint Usage/Research Center on Mathematics and Theoretical Physics).


Mitsuo Izuki,\\
Faculty of Liberal Arts and Sciences, \\
Tokyo City University, \\
1-28-1, Tamadutsumi Setagaya-ku Tokyo 158-8557, Japan. \\ 
E-mail: izuki@tcu.ac.jp

\smallskip

Toru Nogayama,\\
Department of Mathematical Science,\\
Tokyo Metropolitan University,\\
Hachioji, 192-0397, Japan.\\
E-mail: toru.nogayama@gmail.com

\smallskip

Takahiro Noi,\\
Department of Mathematical Science,\\
Tokyo Metropolitan University,\\
Hachioji, 192-0397, Japan.\\
E-mail: taka.noi.hiro@gmail.com

\smallskip

Yoshihiro Sawano (Corresponding author),\\
Department of Mathematical Science,\\
Tokyo Metropolitan University,\\
Hachioji, 192-0397, Japan.\\
+\\
People's Friendship University of Russia.\\
\\
E-mail: yoshihiro-sawano@celery.ocn.ne.jp

\end{document}